\documentclass[11pt]{amsart}
\usepackage{graphicx, color, verbatim}
\usepackage{amscd}
\usepackage{amsmath}
\usepackage{amsfonts}
\usepackage{amssymb}
\usepackage{mathrsfs}
\newtheorem{theorem}{Theorem}[section]

\newtheorem{lemma}[theorem]{Lemma}
\newtheorem{example}[theorem]{Example}

\newtheorem{remark}[theorem]{Remark}
\newtheorem{corollary}[theorem]{Corollary}

\numberwithin{equation}{section}
\newcommand{\RR}{\mathbb R}
\newcommand{\R}{\mathbb R}
\newcommand{\erre}{\mathbb R}

\newcommand{\N}{\mathbb N}
\newcommand{\enne}{\mathbb N}

\renewcommand{\leq}{\leqslant}

\renewcommand{\geq}{\geqslant}

\begin{document}

\title[On nonlocal Dirichlet problems with oscillating term]{On nonlocal Dirichlet problems\\ with oscillating term}

\author[B. Gabrov\v{s}ek]{Bo\v{s}tjan Gabrov\v{s}ek}
\address[Bo\v{s}tjan Gabrov\v{s}ek]{Faculty of Mechanical Engineering, and Faculty of Mathematics and Physics\\
University of Ljubljana \& Institute of Mathematics, Physics and Mechanics\\
1000 Ljubljana, Slovenia}
\email{\tt bostjan.gabrovsek@fs.uni-lj.si}

\author[G. Molica Bisci]{Giovanni Molica Bisci}
\address[Giovanni Molica Bisci]{Dipartimento di Scienze Pure e Applicate (DiSPeA), Universit\`a degli Studi di Urbino
Carlo Bo, Piazza della Repubblica 13, 61029 Urbino (Pesaro e Urbino), Italy}
\email{\tt giovanni.molicabisci@uniurb.it}

\author[D.D. Repov\v{s}]{Du\v{s}an D. Repov\v{s}}
\address[Du\v{s}an D. Repov\v{s}]{Faculty of Education, and Faculty of Mathematics and Physics\\
University of Ljubljana \& Institute of Mathematics, Physics and Mechanics\\
1000 Ljubljana, Slovenia}
\email{\tt dusan.repovs@guest.arnes.si}



\dedicatory{Dedicated to the loving memory of Gaetana Restuccia}
\begin{abstract}
In this paper, a class of nonlocal fractional Dirichlet problems is studied. By using a variational principle due to
 Ricceri
 (whose original version was given
 in  J. Comput. Appl. Math. {113} (2000), 401--410),
the existence of infinitely many weak solutions for 
these problems is
established
by requiring that the nonlinear term $f$ has a suitable oscillating behaviour either at the origin or at infinity.
\end{abstract}

\maketitle 

\section{Introduction}\label{section1}

In the present paper we deal with the following nonlocal fractional problem
\begin{equation}\label{problem}
\left\{ \begin{array}{ll}
(-\Delta)^{s}_{p}u= \lambda \alpha(x)f(u) & \mbox{in} \, \ \ \ \Omega\\
u=0 & \mbox{in} \, \ \ \  \RR^{N}\setminus \Omega, \end{array} \right.
\end{equation}
where $\Omega\subset \RR^N$ is a bounded domain with a smooth (Lipschitz) boundary $\partial\Omega$ and Lebesgue measure $|\Omega|$, $s\in (0, 1)$, $p>N/s$, $\lambda\in \RR$, while $\alpha \in L^{\infty}(\Omega)$ with $\alpha_0:={\rm essinf}_{x\in \Omega}\alpha(x)>0$ and the reaction term $f:\R\to \RR$ is a suitable continuous function. Finally, the leading operator $(-\Delta)^{s}_{p}$ in \eqref{problem} is the degenerate fractional $p$-Laplacian, defined for all $u:\R^N\rightarrow \R$ smooth enough, and
$x \in \R^N$ by
$$
(-\Delta)^{s}_{p}(x):=\displaystyle 2\lim_{\varepsilon\rightarrow 0^+}\int_{\R^{N}\setminus B_\varepsilon(x)}
\frac{|u(x)-u(y)|^{p-2}(u(x)-u(y))}{|x-y|^{N+ps}} dx,
$$
which for $p = 2$ reduces to the linear fractional Laplacian,
 up to a dimensional constant $C(N, s) > 0$; see, for instance,  \cite{Pal11, valpal, FraPa, MRS}.\par

Since elliptic problems involving the fractional $p$-Laplacian operator have been  intensively studied in recent years by several authors, a bibliography list is always far from being complete. To avoid this, we  mention here only the papers \cite{DQ, FMS, I1, I2,I3} and \cite{LL,
MawhinMolica,
Pei, Pu2, Pucci4}, as well as the references therein.\par
Motivated by the wide interest
in problem \eqref{problem}, and in order to treat it, we crucially use
   that, in our setting, the nonlocal fractional Sobolev space
$$X^{s,p}_0(\Omega)=\{u\in W^{s,p}(\R^N): u=0\,\, \mbox{a.e. in}\,\,
\RR^N\setminus
\Omega \},$$
endowed by the norm
$$
\|u\|:=\left(\displaystyle \iint_{\R^{N}\times\R^N}\frac{|u(x)-u(y)|^p}{|x-y|^{N+ps}}\,dxdy\right)^{1/p},
$$
is compactly embedded into the space $C^0(\bar \Omega)$ of continuous functions
 up to the boundary $\partial \Omega$. This regularity result allow us to obtain essential analytical properties of the Euler--Lagrange functional associated to \eqref{problem} in the low-dimensional case; see \cite{Proca,FIK,k} for related topics.\par

Inspired by the results contained in \cite{AM,Proca, OO, OZ2, OZ} and invoking Lemma \ref{EmT}, we then study the number and the asymptotic behavior of the solutions of problem~\eqref{problem}, when $f$ {oscillates near the origin} or {at the infinity}. This analysis is carried
out
 by exploiting variational and topological techniques; see Theorem \ref{BMB} below and \cite[Theorem 2.5]{Ricceri}.\par
 More precisely, fixed $L\in \{0^+,\infty\}$, let
$$
A_L:=\liminf_{t\rightarrow L}\frac{\max_{|\zeta|\leq
t}F(\zeta)}{t^p} \quad\mbox{ and}\quad B_L:=\limsup_{t\rightarrow L
}\frac{F(t)}{t^p},
$$
where $$
F(t):=\displaystyle\int_0^{t}f(\zeta) d\zeta,
\
\hbox{for every}
\
t\in \erre.$$ With the above notations, let us define
\begin{equation*}\label{l1}
\lambda_1^{L}:=\kappa_{p,N,s}\frac{{ \omega_N}}{p \tau^{sp}\alpha_0}\frac{2^N}{B_L}\quad\hbox{ and}\quad
\lambda_2^{L}:=\frac{1}{p\|\alpha\|_\infty |\Omega|K^pA_L},
\end{equation*}
where $
\omega_N
$ denotes the volume of the unit ball in $\R^N,$
$$
\kappa_{p,N,s}:=\displaystyle \frac{{2^{p(3-s)-N}}}{p}\left(1-\frac{1}{2^N}\right)^2+\frac{2^{2+ps-N}}{ps(N+p(1-s))}+\frac{2}{(N-ps)ps}\left(1-\frac{1}{2^{N-ps}}\right),
$$
$$
\tau:=\sup_{x\in\Omega}{\rm dist}(x,\partial\Omega),
\quad
\hbox{and}
\quad
K:=\sup\left\{\frac{\|u\|_\infty}{\|u\|}:u\in X^{s,p}_0(\Omega)\setminus\{0\}\right\}.
$$

The main result reads as follows.

\begin{theorem}\label{Main1}
 Assume that
 $$\displaystyle
 \inf_{t\geq 0}F(t)=0
 \quad
 \hbox{and}
 \quad
\displaystyle { \liminf_{t\rightarrow L}\frac{\max_{|\zeta|\leq
t}F(\zeta)}{t^p}<C \limsup_{t\rightarrow L
}\frac{F(t)}{t^p},}
$$
where $C=C(p,N,s, \alpha,\tau,|\Omega|,K)$ is the geometric constant given by
\begin{equation}\label{Constant}
\displaystyle C:=\left(\frac{\tau^{sp}}{2^N \kappa_{p,N,s}K^p|\Omega|\omega_N}\right)\frac{\alpha_0}{\|\alpha\|_\infty}.
\end{equation}
\indent Then for every $\lambda\in (\lambda_1^{L},\lambda_2^{L}),$ problem \eqref{problem}
admits a sequence $(u_{\lambda,j})_j$ of weak solutions in the fractional Sobolev space $X_0^{s,p}(\Omega)$.\par
Moreover,
$\displaystyle\lim_{j\rightarrow \infty}\|u_{\lambda,j}\|=\infty$
if $L=\infty,$ and
$
\displaystyle \lim_{j\rightarrow \infty}\|u_{\lambda,j}\|=\lim_{j\rightarrow \infty}\|u_{\lambda,j}\|_\infty=0
$
if $L=0^+$.
\end{theorem}

A special and a meaningful case of Theorem \ref{Main1} is the following.

\begin{corollary}\label{Main2}
Assume that $f$ is nonnegative with $f(0)=0$. Furthermore, suppose that
\begin{equation}\label{OMZaIntroinfty}
\displaystyle \liminf_{t\rightarrow L}\frac{F(t)}{t^p}=0\,\,\,\mbox{ and } \,\,\,\displaystyle \limsup_{t\rightarrow L
}\frac{F(t)}{t^p}=\infty.
\end{equation}
Then for every $\lambda>0$ problem \eqref{problem}
admits a sequence $(u_{\lambda,j})_j$ of nonnegative weak solutions in the fractional Sobolev space $X_0^{s,p}(\Omega)$.\par
Moreover,
$\displaystyle\lim_{j\rightarrow \infty}\|u_{\lambda,j}\|=\infty$
if $L=\infty,$ and $
\displaystyle \lim_{j\rightarrow \infty}\|u_{\lambda,j}\|=\lim_{j\rightarrow \infty}\|u_{\lambda,j}\|_\infty=0
$
if $L=0^+$.
\end{corollary}

\indent We notice that the existence of sequences of weak solutions for fractional nonlocal equations, without any symmetry hypothesis on the nonlinear term $f$, has been investigated in \cite[Theorems 5 and 6]{ADM}. However, in the low-dimensional case treated here, it can be easily seen that Theorem \ref{Main1} is
more general than the results proved in the aforementioned paper. We refer to the monograph \cite{MRS} as a general reference for nonlocal problems and variational methods used in this manuscript.

\section{Fractional framework}\label{section2}
This section is devoted to the notations used throughout the paper. In order to give the weak formulation of problem~\eqref{problem}, we need to work in a special functional space. Indeed, one of the difficulties in treating problem~\eqref{problem} is related to encoding the Dirichlet boundary condition in the variational formulation. In this respect, the standard fractional Sobolev spaces are not sufficient in order to study this problem. We overcome this difficulty by working in a new functional space, whose definition is  recalled here.\par
Let $\Omega$ be a bounded domain in $\erre^N$ with smooth (Lipschitz) boundary, fix $s\in (0,1)$ and take $p>N/s$. Let
$$
W^{s,p}(\erre^N):=\left\{u\in L^p(\erre^N):\iint_{\R^{N}\times\R^N}\frac{|u(x)-u(y)|^p}{|x-y|^{N+ps}}\,dxdy<\infty\right\}
$$
be the fractional space endowed with the norm
$$
\|u\|_{s,p}:=\left(\int_{\erre^N}|u(x)|^pdx+\displaystyle \iint_{\R^{N}\times\R^N}\frac{|u(x)-u(y)|^p}{|x-y|^{N+ps}}\,dxdy\right)^{1/p},
u\in W^{s,p}(\erre^N).$$\par
We work on the closed linear subspace defined by
$$X_0^{s,p}(\Omega):=\{u\in W^{s,p}(\erre^N) : u=0\,\, \mbox{a.e. in}\,\,
\RR^N\setminus
\Omega\},$$
and equivalently renormed by setting
$$
\|u\|:=\left(\displaystyle \iint_{\R^{N}\times\R^N}\frac{|u(x)-u(y)|^p}{|x-y|^{N+ps}}\,dxdy\right)^{1/p},
\
u\in X^{s,p}_0(\Omega),$$ namely,
 the Poincar\'{e} inequality holds in $X_0^{s,p}(\Omega)$. The following Rellich-type result will be crucial for our purposes.

\begin{lemma}\label{EmT}
Let $\Omega\subset \RR^N$ be a bounded domain with Lipschitz boundary and let $p >1,$ $s \in (0,1)$ such that $sp > N$.
Then the embedding $$X_0^{s,p}(\Omega)\hookrightarrow C^{0}(\bar\Omega)$$ is compact.
\end{lemma}

\begin{proof}
Since $p>N/s$, by \cite[Theorem 4.47]{Demengel} it follows that $X_0^{s,p}(\Omega)\subset W^{s,p}(\erre^N)$ is continuously embedded into $C^0(\bar\Omega)$; see also \cite[Theorem 8.2]{valpal}.
Now, in order to show that this embedding is also compact, let $B$ be a bounded subset of $X_0^{s,p}(\Omega)$ and let us prove that $B$ is relatively compact in
$C^0(\bar \Omega)$. By virtue of the Arzel\`{a}--Ascoli theorem, the conclusion will be achieved by proving that $B$ is equibounded and equicontinuous on $C^0(\bar\Omega)$. To this end, since $X_0^{s,p}(\Omega)$ by \cite[Theorem 4.47]{Demengel}, is continuously embedded in $C^0(\bar \Omega),$ there exists a constant $c_1>0$
such that
$$
\|u\|_{\infty}\leq c_1 \|u\|,\quad \mbox{for   every} \,  u\in B.
$$
Hence, the set $B$ is equibounded in $C^0(\bar\Omega)$.
Moreover, arguing as in the proof \cite[Theorem 8.2]{valpal}, for every $u\in B,$
the following Morrey-type inequality holds
\begin{equation}\label{Ab}
|u(x)-u(y)|\leq c_2 \|u\|_{s,p}|x-y|^{s-N/p}, \quad \mbox{for  every}\,\,\,  x,y\in\R^N
\end{equation}
for some constant $c_2>0$. Indeed, by formula (8.8) in \cite{valpal}, it follows that
\begin{equation*}
|u(x)-u(y)|\leq c [u]_{p,ps}|x-y|^{s-N/p}, \quad \mbox{for  every}\,\,\,  x,y\in\R^N,
\end{equation*}
where
$$
[u]_{p,sp}:=\left(\sup_{x_0\in \Omega\, \rho>0}\rho^{-sp}\int_{B_\rho(x_0)\cap\Omega}|u(x)-\langle u\rangle_{B_\rho(x_0)\cap\Omega}|^pdx\right)^{1/p},
$$
with
$$
\langle u\rangle_{B_\rho(x_0)\cap\Omega}:=\frac{1}{|B_\rho(x_0)\cap\Omega|}\int_{B_\rho(x_0)\cap\Omega}u(x)dx.
$$
Consequently,
 \eqref{Ab} has been proved by \cite[formula 8.4]{valpal}.
Finally,
\begin{equation}\label{Ac}
\displaystyle \|u\|_{s,p}\leq c_3\left(\iint_{\R^{N}\times\R^N}\frac{|u(x)-u(y)|^p}{|x-y|^{N+ps}}\,dxdy\right)^{1/p},
\
\hbox{for every}
\
u\in B.
\end{equation}
In conclusion, by combining \eqref{Ac} with \eqref{Ab} the equicontinuity of $B$ easily follows. This completes the proof
of Lemma~\ref{EmT}.
\end{proof}

Note that, since the embedding $X_0^{s,p}(\Omega)\hookrightarrow C^{0}(\bar\Omega)$ is continuous, it follows that
$$
K:=\sup\left\{\frac{\|u\|_\infty}{\|u\|}:u\in X^{s,p}_0(\Omega)\setminus\{0\}\right\}<\infty.
$$
It remains an open problem to determine an explicit upper bound for the constant $K$.
\begin{remark}\label{RemarkEmT}
\rm{We note that a more precise version of Lemma \ref{EmT} can be proved by using \cite[Lemma 2.85]{Demengel}. More precisely, if $p>N/s$ and
$C^{0,\mu}_b(\bar\Omega)$ denotes the space of H\"{o}lder continuous
functions of order $\mu$ on $\bar\Omega$, the embedding $X_0^{s,p}(\Omega)\hookrightarrow C^{0,\mu}_b(\bar\Omega)$ is compact provided that $\mu<s-N/p$. The above regularity argument,
with a slight
modification,
 seems
 to work also in anisotropic fractional Sobolev spaces; see \cite{RRR} for related topics.}
\end{remark}

 For further details on the fractional Sobolev spaces we refer to~\cite{valpal, MRS} and to the references therein.\par
Let us fix $\lambda\in \R$. We recall that a \textit{weak solution} for  problem \eqref{problem}, is a function $u:\Omega\to \RR$ such that
\begin{equation}\label{problemaK0fnew}
\left\{\begin{array}{l} {\displaystyle \iint_{\R^{N}\times\R^N }
\frac{|u(x)-u(y)|^{p-2}(u(x)-u(y))(\varphi(x)-\varphi(y))}{|x-y|^{N+ps}} dx dy}\\
\qquad \qquad\qquad \qquad =\lambda \displaystyle\int_\Omega \alpha(x)f(u(x))\varphi(x)dx,
\,\,\,\mbox{for every} \  \varphi \in X_0^{s,p}(\Omega)\\
u\in X_0^{s,p}(\Omega)\,.
\end{array}\right.
\end{equation}

\indent Let $\mathcal J_{\lambda}:X_0^{s,p}(\Omega)\to \RR$ be defined as follows
\begin{equation}\label{Funzionale}
\begin{aligned}
\mathcal J_{\lambda}(u)& :=\frac 1 p \iint_{\R^{N}\times\R^N}\frac{|u(x)-u(y)|^p}{|x-y|^{N+ps}}\,dxdy
-\lambda \int_\Omega \alpha(x)F(u(x))dx,
\end{aligned}
\end{equation}
where, as usual, we set
$$
F(t):=\displaystyle\int_0^{t}f(\zeta)\,d\zeta,
\
\hbox{for every}
\
t\in \erre.$$
Since $f\in C^0(\R,\R)$ and $\alpha\in L^{\infty}(\Omega)$, the functional $\mathcal{J}_{\lambda}\in C^1(X_0^{s,p}(\Omega))$ and its derivative at $u\in X_0^{s,p}(\Omega)$ is given by
$$
\langle \mathcal{J}'_{\lambda}(u), \varphi\rangle = {\displaystyle \iint_{\R^{N}\times\R^N }
\frac{|u(x)-u(y)|^{p-2}(u(x)-u(y))(\varphi(x)-\varphi(y))}{|x-y|^{N+ps}} dxdy}
$$
$$\
-\lambda\int_\Omega \alpha(x)f(u(x))\varphi(x)dx,
\
\hbox{for every}
\
\varphi \in X_0^{s,p}(\Omega).$$\par
 Thus the weak solutions of problem \eqref{problem} are exactly the critical points of the energy functional $\mathcal{J}_{\lambda}$.\par
 Therefore, the proof of the main result reduces to finding critical points of the functional by using suitable abstract approaches.\par

In this direction, we rephrase \cite[Theorem 2.1]{Ricceri} in
a slightly different version; see, for instance, \cite[Theorem 2.1]{Proca}.

\begin{theorem}\label{BMB}
Let $X$ be a reflexive real Banach space and let $\Phi,\Psi:X\to\R$
be two G\^{a}teaux differentiable functionals such that $\Phi$ is
strongly continuous, sequentially weakly lower semicontinuous and coercive, and $\Psi$
is sequentially weakly upper semicontinuous. Set $J_\lambda:=\Phi-\lambda \Psi$. Moreover, for every $r>\displaystyle\inf_X
\Phi,$ put
$$
\varphi(r):=\inf_{u\in\Phi^{-1}((-\infty,r))}\frac{\displaystyle\sup_{v\in\Phi^{-1}((-\infty,r))}\Psi(v)-\Psi(u)}{r-\Phi(u)},
$$
$$
\gamma:=\liminf_{r\to \infty}\varphi(r),\quad
\delta:=\liminf_{r\to (\inf_X \Phi)^+}\varphi(r).
$$
Then one has
\begin{itemize}
\item[$(a)$] If $\gamma<\infty$, then for each
$\lambda\in\left(0,{1}/{\gamma}\right),$ the following
alternatives exist:\\
{\rm $(a_1)$}either  $J_\lambda$ possesses a global minimum,\\
{\rm $(a_2)$}or there is a sequence $(u_j)_{j}$ of critical points
$($local minima$)$ of $J_\lambda$ such that $$\displaystyle\lim_{j\to \infty}
\Phi(u_j)=\infty.$$
\item[$(b)$] If $\delta<\infty$, then for each
$\lambda\in\left(0,{1}/{\delta}\right),$ the following
alternatives exist:\\
{\rm $(b_1)$} either there is a global minimum of $\Phi$ which is a local minimum of $J_\lambda,$\\
{\rm $(b_2)$} or there is a sequence $(u_j)_{j}$ of pairwise distinct
critical points $($local minima$)$ of $J_\lambda$ which weakly
converges to a global minimum of $\Phi,$ with $$\displaystyle\lim_{j\to \infty}
\Phi(u_j)=\inf_{u\in X}\Phi(u).$$
\end{itemize}
\end{theorem}

Following the seminal work of Ricceri \cite{Ricceri}, an impressive number of publications appeared, most of them dedicated to the study of suitable extensions of his variational principle as well as of its consequences; see, for instance, the books \cite{KRV,MoPu,MRS} and the references therein. Recent applications of \cite[Theorem 2.1]{Ricceri} can be found in \cite{DMS,Molica, MoRa}. See also \cite{MoPi} for related topics.

\section{A proof of Theorem \ref{Main1}}\label{}

Let $X:=X^{s,p}_0(\Omega)$ be endowed with the norm
$$
\|u\|:=\left(\displaystyle \iint_{\R^{N}\times\R^N}\frac{|u(x)-u(y)|^p}{|x-y|^{N+ps}}\,dxdy\right)^{1/p},
\
u\in X.$$
Moreover, let $\Phi:X\rightarrow \erre$ and
$\Psi:X\rightarrow \erre$ be defined as follows
$$
\Phi(u):=\frac{1}{p}\displaystyle \iint_{\R^{N}\times\R^N}\frac{|u(x)-u(y)|^p}{|x-y|^{N+ps}}\,dxdy
\quad
\hbox{and}
\quad
\Psi(u):=\int_{\Omega}\alpha(x) F(u(x))dx,$$
\noindent so that if we set $J_\lambda:=\Phi-\lambda\Psi$ as in Theorem \ref{BMB}, then we get
 $\mathcal J_{\lambda}=J_\lambda$.\par
\indent By standard arguments, one shows that $\Phi$ is continuously G\^{a}teaux
differentiable and sequentially weakly lower semicontinuous and
that  its G\^{a}teaux derivative is the functional
$\Phi'(u)\in X^*$ given by
$$
\Phi'(u)(\varphi)={\displaystyle \iint_{\R^{N}\times\R^N }
\frac{|u(x)-u(y)|^{p-2}(u(x)-u(y))(\varphi(x)-\varphi(y))}{|x-y|^{N+ps}} dxdy},
\
\varphi\in X.$$ Moreover, $\Psi$ is continuously G\^{a}teaux differentiable, its G\^{a}teaux derivative is given by
$$
\Psi'(u)(\varphi)=\int_\Omega \alpha(x)f(u(x))\varphi(x)dx,
\
\varphi\in X.$$\par
Now, thanks to Lemma \ref{EmT}, $\Psi$
is a sequentially
weakly continuous functional. Indeed,  for every sequence $(u_j)_{j}$ in
$X$ such that $u_j\rightharpoonup u_0$ for some $u_0\in X$, we shall
prove that
\begin{equation}\label{Zzz}
\lim_{j\rightarrow \infty}\Psi(u_j)=\Psi(u_0).
\end{equation}
Since the embedding $X\hookrightarrow C^{0}(\bar\Omega)$ is compact by Lemma \ref{EmT},
there exists $c>0$ such that $u_{j}\to u_0$ in $C^{0}(\bar\Omega)$ and
\begin{equation}\label{Zz}
\|u_{j}\|_\infty \leq c,
\
\hbox{for every}
\
j\in \N.
\end{equation}
 On the other hand,
$$
\lim_{j\rightarrow \infty}\alpha(x)F(u_{j}(x))=\alpha(x)F(u_{0}(x)),
$$
and, by inequality \eqref{Zz}
$$
|\alpha(x)F(u_{j}(x))|\leq \alpha(x)\max_{|t|\leq c}|F(t)|,
\
\hbox{for a.e.}
\
x\in \bar \Omega
\
\hbox{and every}
\
j\in \N.$$
Hence, since $\alpha\in L^{\infty}(\Omega)$, by using the Lebesgue dominated convergence theorem it follows that
\eqref{Zzz} holds.\par

\indent Now, let $(c_j)_{j}$ be a positive real sequence such
that $$
\displaystyle \lim_{j\rightarrow \infty}c_j=L
\quad
\hbox{and}
\quad
\lim_{j\rightarrow \infty}\displaystyle \frac{\displaystyle\max_{|t|\leq
c_j}F(t)}{c_j^{p}}=A_{L}.
$$
\noindent Set
$$\displaystyle r_j:=\frac{c_j^{p}}{K^{p}p}>0,
\
\hbox{for every}
\
  j \in \N.$$ The continuous embedding $X\hookrightarrow C^{0}(\bar\Omega)$ yields
$$\|v\|_{{\infty}}\leq c_j,
\
\hbox{for every}
\
v\in\Phi^{-1}((-\infty,r_j))
\
\hbox{and}
\
j\in \N.$$ Thus
\begin{eqnarray*}\notag
 \|v\|_{{\infty}}\leq K\|v\|=K(pr_j)^{1/p}=c_j,
 \
\hbox{for every}
\
j\in \N.
\end{eqnarray*}
Then
\begin{eqnarray*}\notag
 \varphi(r_j) &=& \inf_{u\in\Phi^{-1}((-\infty,r_j))}\frac{\displaystyle\sup_{v\in\Phi^{-1}((-\infty,r_j))}\Psi(v)-\Psi(u)}{r_j-\Phi(u)} \\
 &\leq & \frac{\displaystyle\sup_{v\in\Phi^{-1}((-\infty,r_j))}\Psi(v)}{r_j}
\leq
p\|\alpha\|_{\infty}|\Omega|K^{p}\frac{\displaystyle\max_{|t|\leq
c_j}F(t)}{c_j^{p}},
\
\hbox{for every}
\
j\in\N,
\end{eqnarray*}
 taking into account that $\Phi(0)=\Psi(0)=0$.\par
 Hence, if $\lambda<\lambda_2^{L}:=\displaystyle\frac{1}{p\|\alpha\|_\infty |\Omega| K^pA_L}$ it follows that
$$
\displaystyle \beta_L\leq\liminf_{j\rightarrow
\infty}\varphi(r_j)\leq p\|\alpha\|_{\infty}|\Omega|K^{p} A_{L}<\frac{1}{\lambda}<\infty,
$$
where
\begin{equation*}\label{beta}
 \beta_L:=\begin{cases}
  \gamma:=\displaystyle\liminf_{r\to \infty}\varphi(r)\,\,\,\,\,\,\,\,\,\, \hbox{if}\,\, L=\infty \\
 \delta:=\displaystyle\liminf_{r\to 0^+}\varphi(r)\,\,\,\,\,\,\,\,\,\, \hbox{if}\,\, L=0^+.
\end{cases}
\end{equation*}

Let us denote by $B_r(x_0)$ the $N$-dimensional open ball centered at $x_0\in\R^N$ and of radius $r>0$. As $\Omega$ is open, we can certainly choose a point $x_0\in\Omega$ and a number $\tau>0$ so that ${B}_{\tau}(x_0)\subseteq\Omega$.\par If we set $\tau:=\displaystyle\sup_{x\in \Omega}\textrm{dist}(x,\partial \Omega),$ the point $x_0$ is the Chebyshev center of $\Omega$. Hence let us fix such $x_0$ and $\tau$ and define the function $\theta$ to be
\begin{equation}\label{defuetatzero}
\theta(x):=
\left\{
\begin{array}{ll}
0 & \mbox{ if } x \in \R^N \setminus B_{{\tau}}(x_0) \\
\displaystyle 1 & \mbox{ if } x \in B_{{\tau/2}}(x_0)\\
\displaystyle 2\frac{\tau- |x-x_0|}{\tau}  & \mbox{ if } x \in B_{\tau}(x_0)\setminus B_{{\tau}/2}(x_0) \\

\end{array}
\right.
\end{equation}
for every $x\in\R^N$, where $|\cdot|$ denotes the usual Euclidean norm in $\R^N$. Since $\theta \equiv 0$ outside the compact
ball
 $\overline B_{{\tau}}(x_0)$, we can easily deduce that $\theta\in X$.\par
 We now state and prove our main lemma.

\begin{lemma}\label{lemma1}
Let $\theta\in X$ be the function defined in \eqref{defuetatzero}. Then
\begin{equation}\label{stimanormauetatzero}
\displaystyle \int_{\R^{N}\times\R^N}\frac{|\theta(x)-\theta(y)|^p}{|x-y|^{N+ps}}\,dxdy\leq \kappa_{p,N,s}{ \omega^2_N}\tau^{N-sp},
\end{equation}
where
$
\omega_N
$ denotes the volume of the unit ball in $\R^N,$ and
$$
\kappa_{p,N,s}:=\displaystyle \frac{{2^{p(3-s)-N}}}{p}\left(1-\frac{1}{2^N}\right)^2+\frac{2^{2+ps-N}}{ps(N+p(1-s))}+\frac{2}{(N-ps)ps}\left(1-\frac{1}{2^{N-ps}}\right).
$$

\end{lemma}

\begin{proof}
A direct and straightforward computation ensures that
\begin{equation}\label{Sommina}
   \iint_{\R^{N}\times\R^N}\frac{|\theta(x)-\theta(y)|^p}{|x-y|^{N+ps}}\,dxdy=\displaystyle\sum_{j=1}^{4}J_j,
\end{equation}
\noindent  where we set
$$
J_1:=\displaystyle{\int_{B_{\tau}(x_0)\setminus B_{{\tau/2}}(x_0)}\int_{B_{\tau}(x_0)\setminus B_{ {\tau/2}}(x_0)}\frac{|\theta(x)-\theta(y)|^p}{|x-y|^{N+ps}}}\,dxdy,
$$
$$
J_2:={2} \displaystyle{
   \int_{B_{\tau}(x_0)\setminus B_{{\tau/2}}(x_0)}\int_{\R^N\setminus B_{{\tau/2}(x_0)}}\frac{|\theta(x)-\theta(y)|^p}{|x-y|^{N+ps}}
   }\,dxdy,
$$
$$
   J_3:={2}\displaystyle{\int_{B_{{\tau/2}}(x_0)}\int_{B_{\tau}(x_0)\setminus B_{{\tau/2}}(x_0)}\frac{|\theta(x)-\theta(y)|^p}{|x-y|^{N+ps}}}\,dxdy,
   $$
   and
$$
J_4:=\displaystyle 2\displaystyle{\int_{B_{{\tau/2}}(x_0)}\int_{\R^N \setminus B_{\tau}(x_0)}\frac{|\theta(x)-\theta(y)|^p}{|x-y|^{N+ps}}}\,dxdy.
$$
\noindent On the other hand, by virtue of
$$
|y-x_0|-|x-x_0|\leq |x-y|,
$$
and
{\small$$
|x-y|\leq |x-x_0|+|y-x_0|\leq 2\tau,\ \mbox{for all}\, \ (x,y)\in (B_{\tau}(x_0)\setminus B_{ {\tau/2}}(x_0))\times(B_{\tau}(x_0)\setminus B_{{\tau/2}}(x_0)),
$$}one has
$$
\begin{array}{ll}
  J_1 & = \displaystyle\frac{2^p}{\tau^p}\displaystyle{\int_{B_{\tau}(x_0)\setminus B_{ {\tau/2}}(x_0)}\int_{B_{\tau}(x_0)\setminus B_{{\tau/2}}(x_0)}\frac{||y-x_0|-|x-x_0||^p}{|x-y|^{N+ps}}}\,dxdy \\
  & \leq \displaystyle\frac{2^p}{\tau^p}\displaystyle{\int_{B_{\tau}(x_0)\setminus B_{ {\tau/2}}(x_0)}\int_{B_{\tau}(x_0)\setminus B_{{\tau/2}}(x_0)}\frac{|x-y|^p}{|x-y|^{N+ps}}}\,dxdy \\
    &=\displaystyle {2^{p(3-s)-N}}\left(1-\frac{1}{2^N}\right)^2\frac{\tau^{N-ps}}{p}\omega_N^2.\\
\end{array}
$$
 Furthermore, since
 for every
$y\in B_{\tau}(x_0)\setminus B_{{\tau/2}}(x_0),$
$$
\int_{\R^N\setminus B_{\tau}(x_0)}\frac{|\tau -|y-x_0||^p}{|x-y|^{N+ps}}\,dx
=\omega_N\int_{\tau -|y-x_0|}^{\infty}\frac{|\tau -|y-x_0||^p}{\varrho^{ps+1}} d\varrho,
$$
 it follows that
$$
\begin{array}{ll}
  J_2 & = \displaystyle\frac{2^{p+1}}{\tau^p}\displaystyle{\int_{B_{\tau}(x_0)\setminus B_{{\tau/2}}(x_0)}\int_{\R^N\setminus B_{\tau}(x_0)}\frac{|\tau -|y-x_0||^p}{|x-y|^{N+ps}}}\,d x dy \\
   & = \displaystyle\frac{2^{p+1}\omega_N}{\tau^p}\displaystyle{\int_{B_{\tau}(x_0)\setminus B_{{\tau/2}}(x_0)}\int_{\tau -|y-x_0|}^{\infty}\frac{|\tau -|y-x_0||^p}{\varrho^{ps+1}}} d\varrho dy \\
   & = \displaystyle\frac{2^{p+1}\omega_N}{ ps\tau^p}\displaystyle{\int_{B_{\tau}(x_0)\setminus B_{{\tau/2}}(x_0)} |\tau -|x-y||^{p(1-s)} dy} \\
   & = \displaystyle\frac{2^{p+1}\omega_N^2}{ ps\tau^p}\displaystyle{\int_{0}^{\tau/2}z^{N+(1-s)p-1} dz}
    = \displaystyle\frac{2^{1+ps-N}}{s(N+p(1-s))}\frac{\tau^{N-ps}}{p}\omega_N^2,\\
\end{array}
$$
as well as
$$
\begin{array}{ll}
  J_3 & = \displaystyle\frac{2^{p+1}}{\tau^p}\displaystyle{\int_{B_{{\tau/2}}(x_0)}\int_{B_{\tau}(x_0)\setminus B_{{\tau/2}}(x_0)}\frac{||x-x_0|-{\tau/2}|^p}{|x-y|^{N+ps}}}\,dxdy \\
  & = \displaystyle\frac{2^{p+1}}{\tau^p}\displaystyle{\int_{B_{\tau}(x_0)\setminus B_{{\tau/2}}(x_0)}\int_{B_{{\tau/2}}(x_0)}\frac{||x-x_0|-{\tau/2}|^p}{|x-y|^{N+ps}}}\,dydx \\
  & = \displaystyle\frac{2^{p+1}\omega_N}{\tau^p}\displaystyle{\int_{B_{\tau}(x_0)\setminus B_{\tau/2(x_0)}}\left||x-x_0|-\tau/2\right|^p\int_{|x-x_0|-\tau/2}^{|x-x_0|+\tau/2}\frac{d\varrho}{\varrho^{ps+1}}}\, dx \\
   & \leq \displaystyle\frac{2^{p+1}\omega_N}{ps\tau^p}\displaystyle{\int_{B_{\tau}(x_0)\setminus B_{ {\tau/2}}(x_0)}\left||x-x_0|-{\tau/2}\right|^{p(1-s)} }dx \\
   & = \displaystyle\frac{2^{p+1}\omega_N}{ps\tau^p}\displaystyle{\int_{0}^{{\tau/2}} z^{p(1-s)+N-1} dz}
   = \displaystyle\frac{2^{1+ps-N}}{s(N+p(1-s))}\frac{\tau^{N-ps}}{p}\omega_N^2, \\
\end{array}
$$
observing that
$$
\int_{B_{\tau/2}(x_0)}\frac{dy}{|x-y|^{N+ps}} =\omega_N\int_{|x-x_0|-\tau/2}^{|x-x_0|+\tau/2}\frac{d\varrho}{\varrho^{ps+1}},
\
\hbox{for every}
\
x\in B_{\tau}(x_0)\setminus B_{{\tau/2}}(x_0).$$
\noindent Finally,
$$
\begin{array}{ll}
  J_4 & = \displaystyle 2\displaystyle{\int_{B_{{\tau/2}}(x_0)}\int_{\R^N \setminus B_{\tau}(x_0)}\frac{1}{|x-y|^{N+ps}}}\,dxdy \\
  & =  \displaystyle {2} \omega_N\displaystyle{\int_{B_{{\tau/2}}(x_0)}\int_{ \tau-|y-x_0|}^{\infty} \frac{1}{\varrho^{1+ps}}d\varrho} dy
   = \displaystyle\frac{2}{ps} \omega_N\displaystyle{\int_{B_{{\tau/2}}(x_0)} \frac{1}{(\tau-|y-x_0|)^{ps}} dy} \\
   & = \displaystyle\frac{2\omega^2_N}{ps}\displaystyle{\int_{{\tau/2}}^{\tau} z^{N-ps-1}} dz
    = \displaystyle\frac{2}{(N-ps)s}\left(1-\frac{1}{2^{N-ps}}\right)\frac{\tau^{N-ps}}{p}\omega^2_N, \\
\end{array}
$$
due to the fact that
$$
\int_{\R^N \setminus B_{\tau}(x_0)}\frac{1}{|x-y|^{N+ps}}\,dx =\omega_N\int_{\tau -|y-x_0|}^{\infty}\frac{1}{\varrho^{1+ps}} d\varrho,
\
\hbox{for every}
\
y\in B_{{\tau/2}}(x_0).$$ The conclusion follows by \eqref{Sommina} and the above estimates.
\end{proof}
\indent  If $L=\infty,$ 
then we claim that the functional $J_\lambda$ is
unbounded from below. For our goal, let $(\zeta_j)_{j}$ be a real sequence such that
$$\displaystyle \lim_{j\rightarrow
\infty}\zeta_j=\infty$$ and
\begin{equation}\label{B1}
\lim_{j\rightarrow \infty}\displaystyle \frac{F(\zeta_j)}{\zeta_j^{p}}=B_{\infty}.
\end{equation}
 For each $j\in \enne$, let $w_j=\zeta_j\theta\in X$. Then, by Lemma \ref{lemma1},
\begin{equation*}
\begin{aligned}
\Phi(w_j)\leq \kappa_{p,N,s}{ \omega^2_N}\frac{\tau^{N-ps}}{p}\zeta_j^{p},
\
\hbox{for every}
\
j\in \N.
\end{aligned}
\end{equation*}
\par
\indent Moreover, since $\displaystyle\inf_{\zeta\geq 0}F(\zeta)=0$ and $\alpha_0:={\rm essinf}_{x\in \Omega}\alpha(x)>0$, we have
{\small$$
\int_\Omega \alpha(x)F(w_j(x))dx\geq  \alpha_0\int_{B_{\tau/2}(x_0)} F(w_j(x))dx\geq \alpha_0 \frac{\tau^N}{2^N}\omega_NF(\zeta_j)>0,
\
\hbox{for every}
\
j\in\N,$$}since $\theta\equiv 1$ on $B_{\tau/2}(x_0)$. Then, on account of \eqref{stimanormauetatzero},
it follows that
\begin{equation*} \label{2}
J_\lambda(w_j)\leq
\kappa_{p,N,s}{ \omega^2_N}\frac{\tau^{N-ps}}{p}\zeta_j^{p}-\lambda \alpha_0 \frac{\tau^N}{2^N}\omega_NF(\zeta_j),
\
\hbox{for every}
\
j\in \N.
\end{equation*}
\par
\indent If $B_\infty<\infty$, since $\lambda>\lambda^\infty_1,$ we can fix $$\displaystyle \varepsilon\in
\Big(\kappa_{p,N,s}\frac{{ \omega_N}}{\lambda p \tau^{ps}\alpha_0}\frac{2^N}{B_\infty},1\Big).$$
By using (\ref{B1}), 
we get
 $\nu_\varepsilon$ such that
$$
F(\zeta_j)> \varepsilon B_\infty \zeta_j^{p},\,\,\,\, \mbox{for all}\, \ j>\nu_\varepsilon.
$$
\indent Then one has for every $j>\nu_\varepsilon,$ \par
{\small\begin{eqnarray*}\notag
  J_\lambda(w_j)\leq \kappa_{p,N,s}{ \omega^2_N}\frac{\tau^{N-ps}}{p}\zeta_j^{p}-\lambda \alpha_0 \frac{\tau^N}{2^N}\omega_NF(\zeta_j)
 \leq
 \left(\kappa_{p,N,s}\frac{\omega_N}{p\tau^{ps}}-\lambda \varepsilon B_\infty \frac{\alpha_0}{2^N}\right)\tau^N\omega_N\zeta_j^{p}.
\end{eqnarray*}}

 Consequently,  since
$$\displaystyle\lim_{j\rightarrow \infty}\zeta_j=+\infty
\
\hbox{ and}
\
\varepsilon>\kappa_{p,N,s}\displaystyle\frac{\omega_N}{\lambda p\tau^{ps}\alpha_0}\frac{2^N}{B_\infty},$$ it follows that,
$$
\displaystyle \lim_{j\rightarrow \infty}J_\lambda(w_j)=-\infty.
$$
\indent If $B_\infty=\infty$, let us fix $$\displaystyle
M>2^N\kappa_{p,N,s}\frac{{ \omega_N}}{\lambda p\tau^{ps}\alpha_0}.$$
By using again (\ref{B1}), we get $\nu_M$ such that
$$
F(\zeta_j)> M \zeta_j^{p},\,\,\,\, \mbox{for all}\, \  j>\nu_M.
$$
\indent Now we have for every $j>\nu_M,$
{\small\begin{eqnarray*}\notag
J_\lambda(w_j)  \leq \kappa_{p,N,s}{ \omega^2_N}\frac{\tau^{N-ps}}{p}\zeta_j^{p}-\lambda \alpha_0 \frac{\tau^N}{2^N}\omega_NF(\zeta_j)
 =\left(\kappa_{p,N,s}\frac{\omega_N}{p\tau^{ps}}-\lambda M \frac{\alpha_0}{2^N}\right)\tau^N\omega_N\zeta_j^{p}.
\end{eqnarray*}}

 Bearing in mind the choice of $M$ again, we get
$$
\displaystyle \lim_{j\rightarrow \infty}J_\lambda(w_j)=-\infty.
$$
\noindent Therefore, thanks to Theorem \ref{BMB} - Part $(a)$, the functional $J_\lambda$ admits an unbounded
sequence $(u_{\lambda,j})_{j}\subset X$ of critical points.\par
\smallskip
If $L=0^+,$ 
then an argument similar to the one above shows that $u_0\equiv 0$ is not a local minimum point for the functional $J_\lambda$. By Theorem \ref{BMB} - Part $(b)$, the functional $J_\lambda$ admits a sequence $(u_{\lambda,j})_{j}\subset X$ of pairwise distinct
critical points $($local minima$)$ such that
$$
\displaystyle \lim_{j\rightarrow \infty}\|u_{\lambda,j}\|=0.
$$
Finally, by Lemma \ref{EmT} one also have
$\displaystyle\lim_{j\rightarrow \infty}\|u_{\lambda,j}\|_\infty=0$ as claimed.
This completes the proof of Theorem~\ref{Main1}.
\qed

\section{Final comments and remarks}\label{}

Let us give some comments concerning Corollary~\ref{Main2}. To this end, assume that $f:\R\rightarrow \R$ is a nonnegative continuous function with $f(0)=0$.
It is easily seen that Corollary \ref{Main2} can be obtained by applying Theorem \ref{Main1} to the nonlocal problem
\begin{equation}\label{problemX}
\left\{ \begin{array}{ll}
(-\Delta)^{s}_{p}u= \lambda \alpha(x)f^+(u) & \hbox{in} \, \, \Omega\\
u=0 & \hbox{in} \,  \, \RR^{N}\setminus \Omega, \end{array} \right.
\end{equation}
where
$f^+:\R\rightarrow \R,$  defined by
\begin{equation*}\label{fpiu}
 f^+(t):=\begin{cases}
  f(t)\,\,\,\,\,\,\,\, \hbox{if}\,\, t>0 \\
 0\,\,\,\,\,\,\,\,\,\,\,\,\,\,\, \hbox{if}\,\, t\leq 0,
\end{cases}
\end{equation*}
 is continuous for every $t\in \R$.

 Indeed, since $f$ is nonnegative, one has
 $$\displaystyle\max_{|\zeta|\leq t}\int_0^{\zeta}f^+(x)dx=\max_{|\zeta|\leq t}\int_0^{\zeta}f(x)dx=F(t),
 \
 \hbox{for every}
 \
 t\in [0,+\infty),
 $$
  so that the assumptions of Corollary~\ref{Main2} actually give
  $A_L=0$ and $B_L=\infty$.
   Consequently, the main conclusions hold with $\lambda^L_1=0$ and $\lambda^L_2=\infty$.

   Finally, to conclude the proof, we just need to prove that the solutions are nonnegative.
 To this end, let
 $$\xi^{\pm}:=\max\{0,\pm\xi\},
 \
 \hbox{ for every}
 \
 \xi\in \R,$$
 and let $u\in X^{s,p}_0(\Omega)$ be a weak solution of \eqref{problemX}. Then $u^{\pm}\in X^{s,p}_0(\Omega)$ and
\begin{equation}\label{problemXw}
\int_\Omega \alpha(x)f^{+}(u(x))u^{-}(x)dx=0.
\end{equation}
Furthermore, notice that
\begin{equation}\label{problemXww}
|\xi^--\eta^-|^p\leq |\xi-\eta|^{p-2}(\xi-\eta)(\eta^--\xi^-),
\
\hbox{for every}
\
\xi,\eta\in \R.
\end{equation}
 Then, by virtue of \eqref{problemXw} and \eqref{problemXww}, and by testing $J'_\lambda$ with $-u^-\in X^{s,p}_0(\Omega)$, we obtain
$$
0=\langle {J}'_{\lambda}(u), -u^-\rangle = {\displaystyle \iint_{\R^{N}\times\R^N }
\frac{|u(x)-u(y)|^{p-2}(u(x)-u(y))(u^-(x)-u^-(y))}{|x-y|^{N+ps}} dxdy}
$$
$$
\geq {\displaystyle \iint_{\R^{N}\times\R^N }
\frac{|u^-(x)-u^-(y)|^{p-2}}{|x-y|^{N+ps}} dxdy}.
$$
This implies that $u^-$ is constant on $\R^N$ and since $u^-$ vanishes  outside $\Omega$, it follows that $u^-=0$ on the entire space $\R^N$. Thus, $u\geq 0$ a.e.\ in $\Omega$ as claimed.
This completes the proof of Corollary~\ref{Main2}.
\qed

We conclude the paper by an application of Corollary~\ref{Main2}.

\begin{example}\label{esempio}\rm{
Let us consider the following nonlocal fractional problem
\begin{equation}\label{problemFine}
\left\{ \begin{array}{ll}
(-\Delta)^{s}_{p}u= \lambda f(u) & \mbox{in} \, \,   \Omega\\
u=0 & \mbox{in} \, \,   \RR^{N}\setminus \Omega, \end{array} \right.
\end{equation}
where $\Omega\subset \RR^N$ is a bounded domain with smooth (Lipschitz) boundary $\partial\Omega,$ $s\in (0, 1)$, $p>N/s$, and $f:\mathbb{R}\to \mathbb{R}$ is the function defined by
$$
f(t):=\begin{cases}
((k+1)!^p-k!^p)\dfrac{g_k(t)}{\displaystyle\int_{a_k}^{b_k}g_k(\zeta)d\zeta}
&\mbox{if } t\in \bigcup_{k\geq 1}[a_k,b_k]\\
0&\mbox{otherwise},
\end{cases}
$$
where $$
a_k:=\frac{2k!(k+2)!-1}{4(k+1)!} \,\,\,\,\,\,\,\,\,\,\,  \mbox{and}\,\,\,\,\,\,\,\,\,\ b_k:=\frac{2k!(k+2)!+1}{4(k+1)!},
$$
and
$g_k:[a_k,b_k]\rightarrow \mathbb{R}$ is the continuous function given for every $k\geq 1,$ by
$$
g_k(t):=\sqrt{\frac{1}{16(k+1)!}-\left(t-\frac{k!(k+2)}{2}\right)^2}, \quad\mbox{} t\in [a_k,b_k].
$$
By virtue of
Corollary \ref{Main2}, for every
$$
\lambda>\kappa_{p,N,s}\frac{{ \omega_N}}{p \tau^{ps}}{2^{N-p}},
$$
problem \eqref{problemFine} admits a sequence $(u_{\lambda,j})_j\subset X_0^{s,p}(\Omega)$ of (nonnegative) weak solutions such that
$$\displaystyle\lim_{j\rightarrow \infty}\|u_{\lambda,j}\|=\infty.$$
}
\end{example}

\section*{Acknowledgments} The manuscript was realized under the auspices of the Italian MIUR project \textit{Variational methods, with applications to problems in mathematical physics and geometry} (2015KB9WPT 009)
and
the Slovenian Research Agency program P1-0292 and grants N1-0114 and N1-0083.

\end{document}